\documentclass[12pt]{article}

\usepackage{amsfonts,amsmath,amssymb}
\usepackage{amsthm}

\theoremstyle{definition}

\newtheorem{Fact}{Fact}
\newtheorem{Observation}{Observation}

\theoremstyle{plain}
\newtheorem{Theorem}{Theorem}
\newtheorem{Corollary}{Corollary}

\newtheorem{Lemma}{Lemma}

\newcommand{\bfSig}{\mathbf{\Sigma}}

\newcommand{\calN}{\mathcal{N}}

\newcommand{\om}{\omega}
\newcommand{\U}{\mathcal{U}}
\newcommand{\Uom}[1]{\mathcal{U}_{\omega^{#1}}}
\newcommand{\upto}{\upharpoonright}
\newcommand{\con}{\natural}
\newcommand{\wadge}{{W}}

\title{Wadge-like Degrees of Borel bqo-Valued Functions}
\author{Takayuki Kihara\\
{\normalsize Department of Mathematical Informatics, Graduate School of Informatics}\\
{\normalsize Nagoya University, Japan}\\
{\normalsize {\tt kihara@i.nagoya-u.ac.jp}}\\[5pt]
and\\[5pt]
Victor Selivanov\\
{\normalsize A.P. Ershov Institute of Informatics Systems SB RAS}\\
{\normalsize and Kazan  Federal University, Russia}\\
{\normalsize {\tt vseliv@iis.nsk.su}}
}

\begin{document}
\date{}
 \maketitle

\begin{abstract}
We unite two well known generalisations of the Wadge theory. The first one considers more general reducing functions than the continuous functions in the classical case, and the second one extends Wadge reducibility from sets (i.e., $\{0,1\}$-valued functions) to $Q$-valued functions, for a better quasiorder $Q$.
In this article, we consider more general reducibilities on  the $Q$-valued functions and generalise some results of L. Motto Ros in the first direction and of T. Kihara and A. Montalb\'an in the second direction:
Our main result states that the structure of the $\mathbf{\Delta}^0_\alpha$-degrees of $\mathbf{\Delta}^0_{\alpha+\gamma}$-measurable $Q$-valued functions is isomorphic to the $\mathbf{\Delta}^0_\beta$-degrees of $\mathbf{\Delta}^0_{\beta+\gamma}$-measurable $Q$-valued functions, and these are isomorphic to the generalized homomorphism order on the $\gamma$-th iterated $Q$-labeled forests.


\end{abstract}

\section{Introduction}\label{in}

In his thesis \cite{wad84}, W. Wadge introduced a way of measuring the topological complexity of subsets of Baire space $\calN=\om^\om$:
For subsets $A,B$ of $\calN$, we say that $A$ is {\em Wadge reducible} to $B$ ($A\leq_WB$), if $A=f^{-1}(B)$ for some continuous function $f$ on $\calN$.
The definition looks quite elementary; however its surprisingly well-behaved structure was never revealed without the development of deep determinacy techniques.
The induced quotient poset, now called the {\em Wadge degrees}, turns out to be well-founded and have no three pairwise incomparable elements.
Indeed, it provides us an ultimate refinement of all known hierarchies in descriptive set theory, such as the Borel hierarchy and the Hausdorff-Kuratowski difference hierarchy.
Nowadays, the notion of Wadge degrees has become important in several areas including descriptive set theory, inner model theory, computability theory, and automata theory
(see also \cite{cabal12} for the overview of the theory of Wadge degrees).


It is straightforward to generalize the notion of Wadge reducibility to arbitrary topological spaces.
However, contrary to the splendid success in the zero-dimensional case, the Wadge theory for non-zero-dimensional spaces was confronted with serious difficulties, cf.~\cite{ist19}.
Several approaches to solving this difficulty have been proposed, and one of them is considering other natural classes of reducing functions in place of the continuous functions, cf.~\cite{mss15}. 
For a pointclass $\mathbf{\Gamma}$, by {\em $\mathbf{\Gamma}$-function} we mean a function $f$ such that $f^{-1}(A)\in\mathbf{\Gamma}$ for each $A\in\mathbf{\Gamma}$. Since the $\mathbf{\Gamma}$-functions are closed under composition and contain the identity function, we obtain the corresponding $\mathbf{\Gamma}$-reducibility $\leq_\mathbf{\Gamma}$. Among such reducibilities  are $\mathbf{\Delta}^0_\alpha$-reducibilities, for each non-zero countable ordinal $\alpha$. Note that  $\mathbf{\Delta}^0_1$-reducibility coincides with the Wadge reducibility for $\calN$. We  shorten the notation $\leq_{\mathbf{\Delta}^0_\alpha}$ to $\leq_\alpha$, so $\leq_1$ coincides with $\leq_W$. 
The \(\mathbf{\Delta}^0_\alpha \)-functions (which coincide with the \(\mathbf{\Sigma}^0_\alpha \)-functions) and \(
\mathbf{\Delta}^0_\alpha \)-reducibilities (among the much larger class of the so called Borel amenable reducibilities) were  comprehensively investigated by L. Motto Ros (see \cite{mr09} and references therein), and later by \cite{mss15}.
In particular, we have $({\mathbf \Delta}^1_1(\calN);\leq_\alpha)\simeq({\mathbf \Delta}^1_1(\calN);\leq_W)$, $({\mathbf \Delta}^0_\xi(\mathbb{R});\leq_3)\simeq({\mathbf \Delta}^0_\xi(\calN);\leq_3)$, and so on.

Another extension of the Wadge theory is the extension from the case of subsets of $\calN$ to the case of functions $A\colon\calN\to Q$ to an arbitrary quasi-order $Q$.
For $Q=\{0,1,\bot\}$, this extension has already been made by Wadge \cite[Section I.E]{wad84}.
The Wadge hierarchy for the class $Q={\rm Ord}$ of ordinals, known as the hierarchy of norms (cf.~\cite{Blo14}) or the Steel hierarchy (cf.~\cite{du03}), plays a crucial role in descriptive set theory \cite{mos07}.
For the other use, as explicitly described in \cite{km18}, the Wadge theory is strongly tied with Martin's conjecture, one of the most prominent open problems in computability theory (see also the recent Notices article \cite{mon19}), where  in \cite{km18} Wadge reducibility for an arbitrary bqo $Q$ was considered in order to obtain a better understanding on uniform universality for countable Borel equivalence relations.
The necessity of the use of $Q$ was also occurred in \cite{ddw19,ki19}, where (a variant of) the Wadge degrees of $\{0,1,\bot\}$-valued functions was utilized  to analyze the behavior of the structure of real-valued functions, and its connection with the notion of $\alpha$-rank which was originally introduced in J. Bourgain's work on refining the Odell-Rosenthal theorem in Banach space theory.

Hereafter, we identify $Q$-valued functions on a space $X$ with {\em $Q$-partitions of $X$} of the form $\{A^{-1}(q)\}_{q\in Q}$ in order to stress their close relation to $k$-partitions (obtained when $Q=\bar{k}=\{0,\dots,k-1\}$ is an antichain with $k$-elements) studied by several authors.
For $Q$-partitions $A,B$ of $X$, let $A\leq_WB$ mean that there is a  continuous function $f\colon X\to X$ such that $A(x)\leq_QB(f(x))$ for each $x\in X$. The case of sets corresponds to the case of 2-partitions. Let  ${\mathbf \Gamma}(Q^X)$ be the set of $Q$-partitions $A$ of $X$ such that $A^{-1}(q)\in{\mathbf \Gamma}(X)$ for all $q\in Q$. A celebrated theorem of van Engelen, Miller and Steel \cite[Theorem 3.2]{ems87} implies that if $Q$ is a better quasiorder (bqo) then the Wadge ordering $\mathcal{W}_Q=({\mathbf \Delta}^1_1(Q^\mathcal{N});\leq_W)$ on the Borel $Q$-partitions  is a bqo, too (see also Fact \ref{fact:bqo}). Although this theorem gives an important information about the quotient-poset of $\mathcal{W}_Q$, it is  far from a  characterisation.

Many efforts (see e.g. \cite{he93,s07a,s16,s17} and references therein) to characterise the quotient-poset of $\mathcal{W}_Q$ were devoted to {\em $k$-partitions of $\mathcal{N}$}.    The approach in \cite{s07a,s16,s17} to this problem was to characterise the initial segments $({\mathbf \Delta}^0_\alpha(k^\mathcal{N});\leq_W)$ for bigger and bigger ordinals $2\leq\alpha<\omega_1$. To achieve this, the structures of iterated labeled forests with the so called homomorphism quasiorder were defined and useful properties of some natural operations on the iterated labeled forests were discovered, 
which have brought us a fresh look at the deep relationship between Wadge theory and wqo/bqo theory.

An important progress was recently achieved in \cite{km17} where a full characterisation of the quotient-poset of $\mathcal{W}_Q$ for arbitrary bqo $Q$ is obtained, with a heavy use of the (suitably extended)  iterated labeled forests and of the classical computability theory.

In this paper, we unite the above-mentioned extensions of the Wadge theory by characterising the quotient-posets of $({\mathbf \Delta}^1_1(Q^\mathcal{N});\leq_\alpha)$, of their variations for some other Borel amenable reducibilities (which are extended to $Q^\mathcal{N}$ in the obvious way), and of natural initial segments of such quotient-posets.
A typical result (extending the above-mentioned result of L. Motto Ros) may be formulated as follows:
For any bqo $Q$ and any countable ordinals $\alpha>0$ and $\beta\geq 3$, $({\mathbf \Delta}^1_1(Q^\mathcal{N});\leq_W)\simeq({\mathbf \Delta}^1_1(Q^\mathcal{N});\leq_\alpha)\simeq({\mathbf \Delta}^1_1(Q^\mathbb{R});\leq_\beta)$.

We deduce this fact from the following main result (Corollary \ref{mainth-cor}): for all $\alpha,\beta,\gamma<\omega_1$ with $\alpha,\beta>0$ we have $(\mathbf{\Delta}^0_{\alpha+\gamma}(Q^\calN);\leq_\alpha(Q^\calN))\simeq(\mathbf {\Delta}^0_{\beta+\gamma};\leq_\beta)$; which is proved by induction using the results in \cite{km17}. Among particular cases and variations of our result we mention the following: 
 
\begin{enumerate}
\item $({\mathbf \Delta}^0_\omega(Q^\mathcal{N});\leq_n)\simeq({\mathbf \Delta}^0_\omega(Q^\mathcal{N});\leq_W)$ for each $2\leq n<\omega$, 
\item $(\bigcup_{k<\omega}{\mathbf \Delta}^0_k(Q^\mathcal{N});\leq_n)\simeq(\bigcup_{k<\omega}{\mathbf \Delta}^0_k(Q^\mathcal{N});\leq_W)$ for each $2\leq n<\omega$, 
\item $({\mathbf \Delta}^0_5(Q^\mathcal{N});\leq_W)\simeq({\mathbf \Delta}^0_7(Q^\mathcal{N});\leq_3)\simeq({\mathbf \Delta}^0_7(Q^\mathbb{R});\leq_3)\simeq({\mathbf \Delta}^0_{\om+4}(Q^{\mathcal{C}([0,1])});\leq_\om)$.
\end{enumerate}
 
After recalling some preliminaries in the next section, we establish our main result in Section \ref{main} which gives a charactersation of the quotient-posets of $({\mathbf \Delta}^1_1(Q^\mathcal{N});\leq_{1+\xi})$, $\xi<\omega_1$, similar to that in \cite{km17}. Indeed, we prove the same characterisation for some other Borel amenable reducibilities in place of $\leq_{1+\xi}$. Finally, in Section \ref{forest} we provide an inner (i.e., using only notions for labeled forests) characterisations of the quasiorders induced on the iterated labeled forests by the quasiorders $\leq_{1+\xi}$.

\section{Preliminaries}\label{prel}


\subsection{Ordinals, quasi-orders, semilattices}\label{ordinal}

We assume the reader to be acquainted with the notion of ordinal.
Ordinals are denoted by $\alpha, \beta, \gamma,\xi,\eta,\ldots$.
Every non-zero ordinal $\xi$ is uniquely representable in the form $\xi=\omega^{\alpha_0}+\cdots+\omega^{\alpha_n}$ where $n<\omega$ and $\xi\geq\alpha_0\geq\cdots\geq\alpha_n$.
The first
uncountable ordinal is denoted by $\omega_1$.  

We use some standard notation and terminology on partially ordered
sets (posets), cf.~\cite{dp94}.
A {\em quasiorder} (qo) is a reflexive and transitive relation.
%
%
%
A qo is  {\em well-founded} if it has no
infinite descending chains.
A {\em well quasiorder} (wqo) is a qo  that has neither infinite
descending chains nor infinite antichains. Although the wqo's are closed under many natural finitary constructions like forming finite labeled words or trees, they are not always closed under important infinitary constructions.
In 1960s, C. Nash-Williams found a natural subclass of wqo's, called {\em better quasiorders} (bqo's) which contains most of the ``natural'' wqo's (in particular, all finite qo's) and has strong closure properties also for many infinitary constructions.
We omit a bit technical notion of bqo which is used only in formulations.
For more details on bqo's, we refer the reader to \cite{si85}.

By {\em  $\sigma$-semilattice} we mean an (upper) semilattice $(S;\sqcup)$ where supremums $\bigsqcup y_j=y_0\sqcup y_1\sqcup\cdots$ of countable sequences of elements $y_0,y_1,\ldots$ exist. An element $x$ of a $\sigma$-semilattice $S$ is
{\em $\sigma$-join-irreducible} if it cannot be represented as the countable supremum of elements strictly below $x$. As first stressed in \cite{s07}, the $\sigma$-join-irreducible elements play a central role in the study of Wadge degrees of $k$-partitions. The same applies to several variations of Wadge degrees, including the Wadge degrees of $Q$-partitions for a countable bqo $Q$.
\subsection{Descriptive set theory and Wadge-like reducibilities}\label{wadge}

Let $\omega$ be the space of non-negative integers with the
discrete topology. 
By endowing $\calN=\om^\om$ with the product of the discrete
topologies on $\omega$, we obtain the \emph{Baire space}.


We assume the reader to be familiar with Borel hierarchy $\{\mathbf{\Delta}^0_\alpha(X),\bfSig^0_{\alpha}(X),\mathbf{\Pi}^0_\alpha(X)\}_{ \alpha<\omega_1}$ in a Polish space $X$ (see e.g. \cite{ke95,mos07}).
In particular, $\mathbf{\Delta}^1_1(X)=\bigcup\{\mathbf{\Sigma}^0_{1+\alpha}(X)\mid \alpha<\omega_1\}$ is the class of all Borel sets.
Borel hierarchy gives rise to many important classes of functions.
A function $f\colon X\to Y$ between Polish spaces is {\em $\mathbf{\Gamma}$-measurable} if $f^{-1}(U)\in\mathbf{\Gamma}(X)$ for every open set $U\subseteq Y$ (cf.~\cite{ke95,mos07}). The class of such functions is denoted $\mathbf{\Gamma}(X,Y)$.
If the codomain $Y$ is a discrete space, then $\mathbf{\Sigma}^0_\alpha$-measurability coincides with $\mathbf{\Delta}^0_\alpha$-measurability, i.e.~$\mathbf{\Sigma}^0_\alpha(X,Y)=\mathbf{\Delta}^0_\alpha(X,Y)$.
In our proof, we use this notion only when $X=\calN$ and $Y\in\{\calN,Q\}$ where $Q$ is a bqo considered as a discrete topological space.
By discreteness, $\bfSig^0_{\alpha}(\calN,Q)$ coincides with the class of $\mathbf{\Delta}^0_{\alpha}$-partitions of $\calN$  (also denoted $\mathbf{\Delta}^0_{\alpha}(Q^\calN)$) mentioned in Introduction.
By the following topological fact, without loss of generality, one can always assume that $Q$ is countable.

\begin{Fact}[cf.~{\cite[Lemma 9.11 and Remark 9.12]{si85}}]\label{fact:bqo}
For every Borel function $f\colon\calN\to Q$, the image $f(\calN)$ is separable; hence, $f(\calN)$ is countable by discreteness of $Q$.
\end{Fact}

For $\alpha>1$ the class of $\bfSig^0_{\alpha}$-measurable functions on $\calN$ is not closed under composition (hence it does not induce a reasonable degree structure).
On the other hand, there are many natural subclasses of Borel functions closed under composition.
The class $D_\alpha$ of $\mathbf{\Delta}^0_{\alpha}$-functions mentioned in Introduction is such an example:
It is known and easy to see that $D_\alpha$ is closed under composition and contains the identity function, hence the relation $\leq_\alpha$ is always a qo. Furthermore, $D_\alpha\subseteq D_\beta$ for all $0<\alpha<\beta<\omega_1$, hence $\leq_\alpha$ is contained in $\leq_\beta$.

For any pointclass $\mathbf{\Gamma}$ and class $\mathcal{F}$ of functions, we say that a function $f\colon X\to Y$ is {\em $\mathbf{\Gamma}$-piecewise $\mathcal{F}$} if there is a partition $\{X_n\}$ of $X$ to $\mathbf{\Gamma}$-sets and a sequence $f_n\colon X_n\to Y$ of $\mathcal{F}$-functions with $f(x)=f_n(x)$ for any $x\in X_n$.
We denote by $D_\alpha^W$ the class of $\mathbf{\Delta}^0_\alpha$-piecewise continuous functions.
For $\alpha>1$, this class coincides with the $\mathbf{\Sigma}^0_\alpha$-piecewise continuous functions.
Note that $D_\alpha^W\subseteq D_\alpha$, $D_\alpha^W$ is closed under composition and contains the identity function (hence it induces a reducibility $\leq_\alpha^W$ on subsets of $X$). Furthermore, $D^W_\alpha\subseteq D^W_\beta$ for all $0<\alpha<\beta<\omega_1$, hence $\leq^W_\alpha$ is contained in $\leq^W_\beta$.
For more details on $D_\alpha^\wadge$, see \cite{mr09,mss15,ki15,gkn}.

We will use the following basic observations:

\begin{Observation}\label{obs:measu}
\begin{enumerate}
\item Let $\alpha,\beta < \omega_1$, $A\in\bfSig^0_{1+\beta}$, and let $f$ be a $\bfSig^0_{1+\alpha}$-measurable function on $\calN$. Then $f^{-1}(A)\in\bfSig^0_{1+\alpha+\beta}$.
\item For any $\alpha< \omega_1$,  $\bfSig^0_{1+\alpha}(\calN,\calN)\subseteq D_{(1+\alpha)\cdot\omega}$.
\item Let $f$ be a $\mathbf{\Delta}^0_\alpha$-function.
If $g$ is $\mathbf{\Sigma}^0_\alpha$-measurable, so is $g\circ f$.
\end{enumerate}
\end{Observation}

\begin{proof}
All of them are easy to prove. The first observation is well-known (cf.~\cite[Section 24]{ke95}).
For the second one, see also  \cite[Proposition 3.2]{mss15}.
For the last one, $g^{-1}(S)\in\mathbf{\Sigma}^0_{\alpha}$ for any $S\in\mathbf{\Sigma}^0_1$.
Therefore, $(g\circ f)^{-1}(S)=f^{-1}(g^{-1}(S))\in\mathbf{\Sigma}^0_{\alpha}$.
\end{proof}


Along with the classes $D_\alpha$, $D_\alpha^W$ there are other natural classes of reducing functions called Borel amenable classes of functions \cite{mr09}. These classes induce corresponding reducibilities on $Q$-partitions of $\calN$.
As shown in \cite[Proposition 4.3]{mr09}, any Borel amenable class $\mathcal{G}$ is of the form $D_\alpha^\mathcal{F}$ for some $\mathcal{F}$ and $\alpha<\om_1$.
In this case, $\mathcal{G}$ is called a Borel amenable class of level $\alpha$.
As a typical example, given $\beta<\om_1$, let us consider the class $\mathcal{F}$ of $\mathbf{\Sigma}^0_\gamma$-measurable functions for some $\gamma<\beta\cdot\om$.
%
We use the symbol $D_\alpha^{<\beta\cdot\om}$ to denote $D_\alpha^\mathcal{F}$, the class of all $\mathbf{\Delta}^0_{\alpha}$-piecewise $\mathbf{\Sigma}^0_{<\beta\cdot\om}$-measurable functions.
If $\alpha$ is an ordinal of the form $\om^{\alpha_0}+\om^{\alpha_1}+\dots+\om^{\alpha_m}$ for some $\alpha_0\geq\alpha_1\geq\dots\geq\alpha_m$, then we define $\alpha\ast=\om^{\alpha_0}$.
The class $D_{1+\alpha}^{<\alpha\ast}$ naturally arises as seen below.
Moreover, $D_\alpha$ is the largest class among Borel amenable classes of level $\alpha$, cf.~\cite[Section 6]{mr09}.
The main results in \cite{ki15,gkn} clarify the relationship between $D_\alpha$ and $D_\alpha^W$ as follows:
\[D_{1+\alpha}^\wadge\subseteq\dots\subseteq D_{1+\alpha}^{<1+{\alpha\ast}}\subseteq D_{1+\alpha}\subseteq D_{1+\alpha+1}^{<1+{\alpha\ast}}.\]
In \cite{ki15,gkn}, it is conjectured that $D_{1+\alpha}^{<1+\alpha\ast}=D_{1+\alpha}$.
It is not hard to see that if $\alpha$ is of the form $\beta\cdot\om$, then $\alpha{\ast}=\alpha$, so this also provides us a fine picture between the classes mentioned in Observation \ref{obs:measu} (2).

For a class $\mathcal{F}$ of functions, we say that a space $X$ is {\em $\mathcal{F}$-isomorphic to} $Y$ if there is a bijection $f\colon X\to Y$ such that both $f$ and $f^{-1}$ belong to $\mathcal{F}$.
If two spaces have the same $D_\xi$-isomorphism type, they have the same $\leq_{\xi}$-structure of $Q$-partitions:

\begin{Lemma}\label{lem:isomor}
Assume that quasi-Polish spaces $X$ and $Y$ are $D_\xi$-isomorphic.
Then, for any ordinal $\theta\geq\xi$ and qo $Q$, we have $(\mathbf{\Delta}^0_{\theta}(Q^X);\leq_{\xi})\simeq(\mathbf{\Delta}^0_{\theta}(Q^Y);\leq_{\xi})$.
\end{Lemma}

\begin{proof}
Let $h\colon Y\to X$ be a $D_\xi$-isomorphism.
We show that $f\mapsto f\circ h$ induces  $(\mathbf{\Delta}^0_{\theta}(Q^X);\leq_{\xi})\simeq(\mathbf{\Delta}^0_{\theta}(Q^Y);\leq_{\xi})$.
Note that if $f\in\mathbf{\Delta}^0_{\theta}(Q^X)$ then $f\circ h\in\mathbf{\Delta}^0_{\theta}(Q^Y)$ (as $\xi\leq\theta$) by Observation \ref{obs:measu} (3).
For $f,g\colon X\to Q$, assume that $f\leq_{\xi}g$; that is, there is a $\mathbf{\Delta}^0_\xi$-function $\psi$ such that $f(x)\leq_Qg\circ\psi(x)$ for any $x\in X$.
The last inequality is equivalent to the following:
For any $y\in Y$, $f\circ h(y)\leq_Qg\circ\psi\circ h(y)=g\circ h\circ h^{-1}\circ \psi\circ h(y)$; hence $f\circ  h\leq_{\xi}g\circ h$ via $h^{-1}\circ\psi\circ h$, where note that $h^{-1}\circ\psi\circ h$ is a $\mathbf{\Delta}^0_\xi$-function since $D_\xi$ is closed under composition.
For surjectivity, any $g\in\mathbf{\Delta}^0_\theta(Q^Y)$ can be written as $(g\circ h^{-1})\circ h$.
\end{proof}

It is known that every countable dimensional uncountable Polish space is $D^\wadge_3$-isomorphic to $\calN$, cf.~\cite[Theorem 4.21]{mss15}, where a space is {\em countable dimensional} if it is a countable union of finite dimensional subspaces.
Moreover, Kuratowski showed that every uncountable Polish space is $D_\om$-isomorphic to $\calN$, cf.~\cite[Proposition 4.3]{mss15}.
Combining Lemma \ref{lem:isomor} with these facts, we get the following:

\begin{Corollary}\label{cor:more-spaces}
\begin{enumerate}
\item Let $\mathcal{X}$ be a countable dimensional uncountable Polish space.
Then, for any ordinals $\theta\geq\xi\geq 3$,  we have $(\mathbf{\Delta}^0_{\theta}(Q^\mathcal{X});\leq_{\xi})\simeq(\mathbf{\Delta}^0_{\theta}(Q^\calN);\leq_{\xi})$.
\item Let $\mathcal{X}$ be an uncountable Polish space.
Then, for any ordinals $\theta\geq\xi\geq\om$, we have $(\mathbf{\Delta}^0_{\theta}(Q^\mathcal{X});\leq_{\xi})\simeq(\mathbf{\Delta}^0_{\theta}(Q^\calN);\leq_{\xi})$.
\end{enumerate}
\end{Corollary}

\subsection{Tree calculus}\label{trees}

Let $\omega^*$ be the set of finite sequences
of elements of $\omega$, including the empty sequence $\varepsilon$.
For $\sigma,\tau\in\omega^*$, we write
$\sigma\sqsubseteq \tau$ to denote that $\sigma$ is an initial
segment of the sequence $\tau$.
A {\em tree} is a non-empty set $T\subseteq\omega^*$ which is closed downwards under $\sqsubseteq$.
For any qo $Q$, a {\em $Q$-tree} is a pair $(T,t)$ consisting of a well founded tree $T\subseteq\omega^*$ and a labeling $t\colon T\to Q$. Let $\mathcal{T}(Q)$ be the set of $Q$-trees quasi-ordered by the relation: $(T,t)\leq_h(V,v)$ iff there is a monotone function $\varphi\colon T\to V$ with $\forall x\in T(t(x)\leq_Qv(\varphi(x)))$. Let $\mathcal{T}^\sqcup(Q)$ be defined similarly but with forests (i.e., the sub-qo $T\setminus\{\varepsilon\}$ of a tree $T$) in place of trees. As follows from Laver's theorem, if $Q$ is bqo then so are also $\mathcal{T}(Q)$ and $\mathcal{T}^\sqcup(Q)$, hence $\mathcal{T}$ is an operator on the class BQO of all bqo's. This operator was introduced in \cite{s07} and turned out useful (together with some of its iterates) for characterising some initial segments of $\mathcal{W}_Q$ (we warn the reader that operator $\mathcal{T}$ is denoted $\widetilde{\mathcal{T}}$ in \cite{s07}).
As observed in \cite{s07,s16},  $\mathcal{T}^\sqcup(Q)$ is a $\sigma$-semilattice (the countable supremum operation $\bigsqcup$ is the disjoint union of labeled forests). The $\sigma$-join-irreducible elements of $\mathcal{T}^\sqcup(Q)$ are precisely those $h$-equivalent to the elements of $\mathcal{T}(Q)$.

Next we recall iterations of $\mathcal{T}$ from \cite{s16,s17}.
For any $q\in Q$, let $s(q)$ be the singleton tree labeled by $q$, then $s\colon Q\to\mathcal{T}(Q)$ is an embedding of qo's. Identifying $q$ with $s(q)$, we may think that $Q$ is a substructure of $\mathcal{T(Q)}$. We iterate the operator $\mathcal{T}$ as follows: $\mathcal{T}^0(Q)=Q$, $\mathcal{T}^{\alpha+1}(Q)=\mathcal{T}(\mathcal{T}^\alpha(Q))$, and $\mathcal{T}^\lambda(Q)=\bigcup_{\alpha<\lambda}\mathcal{T}^\alpha(Q)$ for a limit ordinal $\lambda$. Then $\{\mathcal{T}^\alpha(Q)\}$ is an increasing sequence of bqo's. Since all our trees are countable, $\mathcal{T}^{\omega_1}(Q)$ is a fixed point of this iteration procedure.
The function $s$ is naturally extended to a function $s$ on $\mathcal{T}^{\omega_1}$ such that $q=s(q)$, $T\leq_hs(T)$, and $T\leq_hV$ iff $s(T)\leq_hs(V)$.
This iteration procedure was extended in \cite{km17} by considering  operations $(s_\alpha)_{\alpha<\omega_1}$ in place of just one function $s$ (to unify and simplify notation, we use the notation $s_\alpha(T)$ instead of the notation $\langle T\rangle^{\omega^\alpha}$ in \cite{km17}).
The idea of the iteration may be described as follows:
First take the ($\om_1$st) fixed point $\mathcal{T}_\om:=\mathcal{T}^{\om_1}$ closed under $s_0=s$; then add new $s_1$ which enumerates fixed points for $s_0$ in the sense that $s_0(s_1(T))\equiv_hs_1(T)$, and take the ($\om_1$st) fixed point $\mathcal{T}_{\om^2}$ closed under $s_0$ and $s_1$.
Continue this Veblen-like procedure to produce $(\mathcal{T}_{\om^\alpha};s_\alpha)_{\alpha<\om_1}$.

We now give the precise inductive definition of $(\mathcal{T}_{\om^\alpha},\mathcal{T}_{\om^\alpha}^\sqcup)_{\alpha<\om_1}$ formalizing the idea described in the previous paragraph.
In \cite[Definition 3.19]{km17}, $\mathcal{T}^\sqcup_{\om^\alpha}(Q)$ is defined as a set of terms in the language consisting of constant symbols $s_\alpha(q)$ for $q\in Q$, a $2$-ary function symbol $\cdot$, an $\om$-ary function symbol $\sqcup$, and unary function symbols $s_\beta$ for $\beta<\alpha$:
Every constant symbol is a {\em singleton term}, and every singleton term is a {\em tree term}.
If $(S_i)_{i\in\om}$ is a sequence of tree terms, then $\sqcup_iS_i$ is a {\em forest term}.
If $S$ is a singleton term and $F$ is a forest term, then $S\cdot F$ is a tree term.
If $T$ is a tree term, then $s_\beta(T)$ is a singleton term for any $\beta<\alpha$.
Then $\mathcal{T}_{\om^\alpha}(Q)$ is the set of all tree terms, and $\mathcal{T}^\sqcup_{\om^\alpha}(Q)$ is the set of all tree and forest terms.
%
For any non-zero countable ordinal $\xi=\omega^{\alpha_0}+\cdots+\omega^{\alpha_n}$, $\alpha_0\geq\cdots\geq\alpha_n$, we define the operator $\mathcal{T}_\xi=\mathcal{T}_{\omega^{\alpha_0}}\circ\cdots\circ\mathcal{T}_{\omega^{\alpha_n}}$   (let also $\mathcal{T}_0$ be the identity operator on BQO).
Finally, let $\mathcal{T}_{\omega_1}(Q)=\bigcup_{\xi<\omega_1}\mathcal{T}_{\xi}(Q)$, and similarly for $\mathcal{T}^\sqcup_{\omega_1}(Q)$.

As in \cite[Definition 3.20]{km17}, we inductively define a qo $\leq_h$ on $\mathcal{T}_{\om_1}^\sqcup(Q)$ as follows:
For $p,q\in Q$ and $\alpha<\om_1$, $p\equiv_hs_\alpha(p)$, and $p\leq_hq$ iff $p\leq_Qq$.
For singletons $s_\alpha(U)$ and $s_\beta(V)$, $s_\alpha(U)\leq_hs_\beta(V)$ is equivalent to $U\leq_h V$ if $\alpha=\beta$; to $s_\alpha(U)\leq_hV$ if $\alpha>\beta$; and to $U\leq_hs_\beta(V)$ if $\alpha<\beta$.
For singletons $A,C$ and forests $B,D$ (where the empty forests are allowed, cf.~\cite[Definition 3.20]{km17}), define $A\cdot B\leq_h C\cdot D$ if either $A\leq_h C$ and $B\leq_h C\cdot D$, or $A\not\leq_hC$ and $A\cdot B\leq_h D$.
Moreover, $B=\sqcup_iB_i\leq_h C\cdot D$ iff $B_i\leq_hC\cdot D$ for any $i$, and $A\cdot B\leq_hD=\sqcup_iD_i$ iff $A\cdot B\leq_hD_i$ for some $i$.
Again,  $\mathcal{T}^\sqcup_{\omega_1}(Q)$ is a $\sigma$-semilattice the $\sigma$-join-irreducible elements of which are precisely those $h$-equivalent to the elements of $\mathcal{T}_{\omega_1}(Q)$.
See also \cite{s18}.


\subsection{Characterising Wadge degrees}\label{charW}

As an extension on a number of previous works on the Wadge degrees, the complete characterisation of the $Q$-Wadge degrees $\mathcal{W}_Q$ in terms of the iterated labeled forests is described in \cite{km17} as follows.

\begin{Theorem}[\cite{km17}]\label{char1}
Let $\eta<\omega_1$ and $Q$ be a bqo. Then  $(\mathcal{T}^\sqcup_{\eta}(Q);\leq_h)\simeq(\mathbf{\Delta}^0_{1+\eta}(Q^\mathcal{N});\leq_W)$ and $(\mathcal{T}^\sqcup_{\omega_1}(Q);\leq_h)\simeq(\mathbf{\Delta}^1_1(Q^\mathcal{N});\leq_W)$. 
 \end{Theorem}

The basic strategy of the proof is as follows:
First assign a natural class $\Sigma_T$ of functions to each $T\in\mathcal{T}^\sqcup_{\om_1}(Q)$, which refines known hierarchies such as the Borel hierarchy $(\mathbf{\Sigma}^0_\alpha)_{\alpha<\om_1}$ and the difference hierarchy $(\mathcal{D}_\beta(\mathbf{\Sigma}^0_\alpha))_{\alpha,\beta<\om_1}$.
The main task is to show that the hierarchy $(\Sigma_T)_{T\in\mathcal{T}^\sqcup_{\om_1}(Q)}$ is {\em ultimate} in the sense that there is no finer hierarchy of Borel functions from the viewpoint of continuous reducibility.
To achieve this, to each $T\in\mathcal{T}^\sqcup_{\om_1}(Q)$ associate a {\em $\Sigma_T$-complete} function $\mu(T)\colon\calN\to Q$ (the symbol $\Omega_T$ is used in \cite{km17}).


In order to define $\mu$, the next step of the proof strategy is to overcome the difficulty caused by the non-existence of a universal total $\mathbf{\Sigma}^0_{1+\xi}$-measurable function, and the notion of a conciliatory function was designed to solve this problem (cf.~\cite{du01,km17}).

Let $\con\colon\calN\to\calN$ be a function with $\con\circ\con=\con$.
We say that a function $f\colon\calN\to\calN$ is {\em $\con$-conciliatory} if, for any $x,y\in\calN$, $\con(x)=\con(y)$ implies $\con\circ f(x)=\con\circ f(y)$.
Similarly, a function $A\colon\calN\to Q$ is {\em $\con$-conciliatory} if, for any $x,y\in\calN$, $\con(x)=\con(y)$ implies $A(x)=A(y)$.
We say that $f,g\colon\calN\to\calN$ are {\em $\con$-equivalent} (written $f\equiv_\con g$) if $\con\circ f=\con\circ g$.
We will frequently use the following basic observation:
If $A\colon\calN\to Q$ is $\con$-conciliatory, and $f,g\colon\calN\to\calN$ are $\con$-equivalent, then $A\circ f=A\circ g$.
For $Q=2$, this notion was first introduced by \cite{du01}.
See also \cite[Sections 2.5 and 2.6]{km17} for the idea behind these definitions.
Note that, in order to avoid going back and forth between two spaces $\calN$ and $\hat{\om}^\om$ as in \cite{km17}, our definitions are slightly different from the original one ($\con$ plays a similar role to ${\sf p}$ in \cite{km17} as seen below).

If we suitably choose $\con$, then the class of $\con$-conciliatory functions has the following good property \cite{km17}:

\begin{Fact}\label{fact:conciliatory}
\begin{enumerate}
\item Any partial continuous function $g$ on $\calN$ has a {\em conciliatory total extension}; that is, there is a $\con$-conciliatory total continuous function $\hat{g}\colon\calN\to\calN$ such that $\con\circ g(x)=\con\circ\hat{g}(x)$ holds for any $x\in{\rm dom}(g)$.
\item For any countable ordinal $\xi$, there is a $\mathbf{\Sigma}^0_{1+\xi}$-measurable $\con$-conciliatory function $\mathcal{U}_\xi\colon\calN\to\calN$ which is {\em universal}; that is, for every $\mathbf{\Sigma}^0_{1+\xi}$-measurable function $f\colon\calN\to\calN$, there is a continuous function $g\colon\calN\to\calN$ such that $f$ is $\con$-equivalent to $\U_\xi\circ g$.
\item Every $\sigma$-join-irreducible Borel function $f\colon\calN\to Q$ is Wadge equivalent to a $\con$-conciliatory function.
Indeed, for any tree $T\in\mathcal{T}_{\om_1}(Q)$, there is a $\Sigma_T$-complete $\con$-conciliatory function $\mu(T)\colon\calN\to Q$.
\end{enumerate}
\end{Fact}
For (1), see \cite[Observation 2.19]{km17}; for (2), see \cite[Proposition 2.15 and Proposition 3.24]{km17}; and for (3), see \cite[Observation 3.15]{km17}.
Regarding Fact \ref{fact:conciliatory} (3), even if $F$ is not a tree, one can ensure that $\mu(F)$ is {\em almost $\con$-conciliatory}; that is, $\con(x)=\con(y)$ implies $\mu(F)(nx)=\mu(F)(ny)$ for any $n\in\om$ and $x,y\in\calN$.

To define $\mu(T\cdot F)$ we need a conciliatory Wadge addition operation, which is denoted as $A^\to B$ in \cite{du01,km17}, but we use the symbol $A\cdot B$ in this paper.
We do not mention the explicit definition of the operation $A\cdot B$, because we only use the following special properties:

\begin{Fact}\label{fact:concate}
If $A$ is $\con$-conciliatory, and $B$ is almost $\con$-conciliatory, then $A\cdot B$ is $\con$-conciliatory.
Moreover, there are $\con$-conciliatory continuous functions $\pi_0,\pi_1\colon\calN\to\calN$ and an open set $J\subseteq\calN$ such that if $x\in J$ then $(A\cdot B)(x)=B\circ\pi_1(x)$, and if $x\not\in J$ then $(A\cdot B)(x)=A\circ\pi_0(x)$.
\end{Fact}
For this fact, see \cite[Observation 3.11]{km17}.
Let us briefly explain how to obtain such a function $\con$ since our notation is slightly different from \cite{km17}:
For a homeomorphism $I\colon(\om\cup\{{\sf pass}\})^\om\to\calN$ and a map $z\mapsto z^{\sf p}\colon(\om\cup\{{\sf pass}\})^\om\to\om^{\leq\om}$ in \cite[Section 2.5]{km17}, consider $\delta(x)=x\mapsto (I^{-1}(x))^{\sf p}$, which is a total surjection from $\calN$ to $\om^{\leq\om}$, and we define $\natural(x)=I(\delta(x){\sf pass}^\om)$.
Then, $\natural$-equivalence is the same as $\equiv_{\sf p}$ in \cite[Section 2.6]{km17}, so this $\con$ works.
Hereafter, we fix such a function $\con$ satisfying Facts \ref{fact:conciliatory} and \ref{fact:concate}, and we never use the explicit definition of $\con$.
We also use the terminology {\em conciliatory} instead of $\con$-conciliatory.

For $T\in\mathcal{T}_{\om_1}^\sqcup(Q)$, the definition of $\mu$ proceeds by induction on the rank of $T$ (which is defined by induction scheme from the end of the previous section) so that the following holds: if $T\equiv_hs_\alpha(q)$ for $q\in Q$ then $\mu(T)$ is the constant function $\lambda x.q$ on ${\mathcal{N}}$; if $T=s_\alpha(V)$ for some $V$ distinct (modulo $\equiv_h$) from all $q\in Q$ then $\mu(T)=\mu(V)\circ\mathcal{U}_{\omega^\alpha}$;
if $T=\sqcup_iS_i$ then $\mu(T)=\bigoplus_i\mu(S_i)$, where $(\bigoplus_iS_i)(nx)=S_n(x)$.
if $T=S\cdot F$ for some tree $S$ and forest $F$ then $\mu(T)=\mu(S)\cdot\mu(F)$.
This definition fulfills Fact \ref{fact:conciliatory} (3), and also the following fact by \cite[Lemma 3.9, Observation 3.16, and Lemma 3.22]{km17}.

\begin{Fact}\label{fact:mu-measu}
For any ordinal $\xi$, if $T\in\mathcal{T}^\sqcup_\xi(Q)$ then $\mu(T)\in\mathbf{\Delta}^0_{1+\xi}(Q^\calN)$.
\end{Fact}

Isomorphisms between the quotient posets from Theorem \ref{char1} are induced by this function $\mu\colon\mathcal{T}^\sqcup_{\omega_1}(Q)\to\mathbf{\Delta}^1_1(Q^{\mathcal{N}})$.
Namely, $\mu$ is an embedding \cite[Proposition 1.7]{km17}, and surjective \cite[Proposition 1.9]{km17} in the following sense:
\begin{Fact}[\cite{km17}]\label{fact:km-main}
For any $T,V\in\mathcal{T}^\sqcup_{\omega_1}(Q)$, $T\leq_hV$ if and only if $\mu(T)\leq_W\mu(V)$.
For any $\mathbf{\Delta}^0_{1+\eta}$-function $A\colon \calN\to Q$ there is $T\in\mathcal{T}^\sqcup_{\eta}(Q)$ such that $A\equiv_W\mu(T)$.
\end{Fact}

\section{Main result} \label{main}

In this section we formulate and prove the main result of this paper for the reducibilities $\leq_{1+\xi}$, $\xi<\omega_1$.
This result which clearly implies all the results mentioned in Introduction, is formulated as follows.

\begin{Theorem}\label{mainth}
For all  $\xi,\eta<\omega_1$ and bqo $Q$,
\[(\mathcal{T}^\sqcup_{\eta}(Q);\leq_h)\simeq(\mathbf{\Delta}^0_{1+\xi+\eta}(Q^\mathcal{N});\leq_{1+\xi})\simeq(\mathbf{\Delta}^0_{1+\xi+\eta}(Q^\mathcal{N});\leq^\wadge_{1+\xi}).\]  
\end{Theorem}

\begin{Corollary}\label{mainth-cor}
For all  $\alpha,\beta,\eta<\omega_1$ and bqo $Q$,
\[(\mathbf{\Delta}^0_{1+\alpha+\eta}(Q^\mathcal{N});\leq_{1+\alpha})\simeq(\mathbf{\Delta}^0_{1+\beta+\eta}(Q^\mathcal{N});\leq_{1+\beta}).\]  
\end{Corollary}

By Corollary \ref{cor:more-spaces}, our results extend to various spaces other than Baire space $\calN$.
For instance, if $X$ is $d$-dimensional Euclidean space $\mathbb{R}^d$, then for all $\alpha,\beta,\eta<\omega_1$ with $\alpha,\beta\geq 2$, we have $(\mathcal{T}^\sqcup_{\eta}(Q);\leq_h)\simeq(\mathbf{\Delta}^0_{1+\alpha+\eta}(Q^{X});\leq_{1+\alpha})\simeq(\mathbf{\Delta}^0_{1+\beta+\eta}(Q^{X});\leq_{1+\beta})$.
Similarly, if $X$ is e.g.~Hilbert cube $[0,1]^\mathbb{N}$, infinite dimensional  separable Hilbert space $\ell_2$, or the function space $\mathcal{C}([0,1],\mathbb{R})$, then for all $\alpha,\beta,\eta<\omega_1$ with $\alpha,\beta\geq\om$, we have $(\mathcal{T}^\sqcup_{\eta}(Q);\leq_h)\simeq(\mathbf{\Delta}^0_{1+\alpha+\eta}(Q^{X});\leq_{1+\alpha})\simeq(\mathbf{\Delta}^0_{1+\beta+\eta}(Q^{X});\leq_{1+\beta})$.




The proof of main result proceeds by induction on $\xi$. Note that for $\xi=0$ the assertion coincides with Theorem \ref{char1}. Let us explain what happens in the simplest case $\eta=0$.

First we observe that any qo induces a kind of free $\sigma$-semilattice $Q^\sqcup$  which we define as the qo $(Q^*;\leq^*)$ where $Q^*$ is the set of non-empty countable subsets of $Q$ with the so called {\em domination qo} defined by $S\leq^*R$ iff $\forall s\in S\exists r\in R(s\leq_Q r)$. Note that the operation $\bigsqcup$ of countable supremum in $Q^\sqcup$ is induced by the operation of countable union in $Q^*$ (Categorical properties of $Q\mapsto Q^\sqcup$ and characterisation of some algebras expanding $(\mathcal{T}_{\om^\alpha}^\sqcup;\leq_h)$ as free structures are considered in \cite{s18}).

Since $\mathcal{T}_0$ is the identity operator, the case $\eta=0$ reduces to the following assertion.

\begin{Lemma}\label{eta0}
 For any bqo $Q$, $Q^\sqcup\simeq(\mathbf{\Delta}^0_{1+\xi}(Q^\mathcal{N});\leq_{1+\xi})\simeq(\mathbf{\Delta}^0_{1+\xi}(Q^\mathcal{N});\leq_{1+\xi}^\wadge)$.
\end{Lemma} 

\begin{proof}
We have to check that $(\mathbf{\Delta}^0_{1+\xi}(Q^\mathcal{N});\leq_{1+\xi})\simeq Q^*$.
Associate with any  $A\in\mathbf{\Delta}^0_{1+\xi}(Q^\mathcal{N})$ the image $A(\calN)\in Q^*$.
Observe that, if $A\leq_{1+\xi}B$ via $f$ then $A(x)\leq_Q B(f(x))$, hence $A(\calN)\leq^*B(\calN)$.
Conversely, let $A(\calN)\leq^*B(\calN)$, then for some $g\colon A(\calN)\to B(\calN)$ we have $q\leq_Q g(q)$ for any $q\in A(\calN)$.
Given $q\in A(\calN)$, choose $y_q\in\calN$ with $B(y_q)=g(q)$. Now define a function $f$ on $\calN$ by $f(x)=y_{A(x)}$.
Then $A(x)\leq_Q B(f(x))$.
Note that the image of $f$ is countable as $A(\calN)$ is countable by Fact \ref{fact:bqo}.
Moreover, for each $y\in\calN$, since $f^{-1}(y)$ is the union of some sets $A^{-1}(q)\in\mathbf{\Delta}^0_{1+\xi}$, $q\in A(\calN)$, and $A(\calN)\subseteq Q$ is countable, we have $f^{-1}(y)\in\mathbf{\Sigma}^0_{1+\xi}$.
As the image of $f$ is countable, this means that $f$ is $\mathbf{\Sigma}^0_{1+\xi}$-piecewise constant.
Since $\mathbf{\Sigma}^0_{1+\xi}$-piecewise continuity is clearly equivalent to $\mathbf{\Delta}^0_{1+\xi}$-piecewise continuity, this shows that $A\leq_{1+\xi}^\wadge B$.
\end{proof}
 

For any $\xi<\omega_1$ we define $s^*_\xi\colon\mathcal{T}_{\omega_1}(Q)\to\mathcal{T}_{\omega_1}(Q)$ as follows: let $s^*_0$ be the identity transformation, and for non-zero $\xi=\omega^{\alpha_0}+\cdots+\omega^{\alpha_m}$, $\alpha_0\geq\cdots\geq\alpha_m$, we set $s^*_\xi=s_{\alpha_0}\circ\cdots\circ s_{\alpha_m}$.
It is straightforward to check the following:

\begin{Observation}\label{t}
\begin{enumerate}
\item For all $\xi<\omega_1$, bqo $Q$, and $T,V\in\mathcal{T}_{\omega_1}(Q)$, we have: $T\leq_hV$ iff $s^*_\xi(T)\leq_hs^*_\xi(V)$.
\item For all $\xi,\eta<\omega_1$ and bqo $Q$, $s^*_\xi$ maps $\mathcal{T}_{\eta}(Q)$ into $\mathcal{T}_{\xi+\eta}(Q)$.
\end{enumerate}
\end{Observation}

%

 
Recall from \cite{km17} that the $\mathbf{\Sigma}^0_{1+\xi}$-universal conciliatory function $\mathcal{U}_\xi$ for a non-zero $\xi=\omega^{\alpha_0}+\cdots+\omega^{\alpha_m}$ coincides with $\mathcal{U}_{\omega^{\alpha_m}}\circ\cdots\circ\mathcal{U}_{\omega^{\alpha_0}}$ where $\mathcal{U}_{\omega^{\alpha_i}}$ is a $\mathbf{\Sigma}^0_{1+\omega^{\alpha_i}}$-universal conciliatory function in Fact \ref{fact:conciliatory} (2).
We show that $T\mapsto \mu s^\ast_\xi(T)$ induces an embedding of $(T_\eta(Q);\leq_h)$ into $(\mathbf{\Delta}^0_{1+\xi+\eta}(Q^\calN);\leq_{1+\xi})$, where recall Fact \ref{fact:mu-measu} for the range of $\mu s^\ast_\xi$.

\begin{Lemma}\label{mut}
For all $\xi,\eta<\omega_1$, bqo $Q$, and $T,V\in\mathcal{T}_{\eta}(Q)$ we have: $T\leq_hV$ iff  $\mu s^*_\xi(T)\leq_1\mu s^*_\xi(V)$ iff $\mu s^*_\xi(T)\leq_{1+\xi}\mu s^*_\xi(V)$.
\end{Lemma}

\begin{proof}
For the first equivalence, by Observation \ref{t}, we have $T\leq_hV$ if and only if $s_\xi^\ast(T)\leq_hs_\xi^\ast(V)$.
By Fact \ref{fact:km-main}, the latter is equivalent to $\mu s^*_\xi(T)\leq_1\mu s^*_\xi(V)$.
The second equivalence for $\xi=0$ follows from Observation \ref{t}. For the second equivalence for $\xi>0$, it suffices to show that $\mu s^*_\xi(T)\leq_{1+\xi}\mu s^*_\xi(V)$ implies $\mu s^*_\xi(T)\leq_1\mu s^*_\xi(V)$. Let $f$ be a $\mathbf{\Delta}^0_{1+\xi}$-function such that $ \mu s^*_\xi(T)=\mu s^\ast_\xi(V)\circ f$. Then
\[
\mu s^\ast_\xi(V)\circ f
= \mu s_{\alpha_0}\cdots s_{\alpha_m}(V)\circ f
= \mu(T)\circ \mathcal{U}_{\omega^{\alpha_m}}\circ\cdots\circ\mathcal{U}_{\omega^{\alpha_0}}\circ f
=\mu(T)\circ \mathcal{U}_\xi\circ f.
\]
As $f\in D_{1+\xi}$ and $\U_\xi$ is $\mathbf{\Sigma}^0_{1+\xi}$-measurable, by Observation \ref{obs:measu} (3), $\mathcal{U}_\xi\circ f$ is $\mathbf{\Sigma}^0_{1+\xi}$-measurable.
Since $\mathcal{U}_\xi$ is $\mathbf{\Sigma}^0_{1+\xi}$-universal by Fact \ref{fact:conciliatory} (2), there is a continuous function $g$ on such that $\mathcal{U}_\xi\circ f$ is $\con$-equivalent to $\mathcal{U}_\xi\circ g$.
By Fact \ref{fact:conciliatory} (3), as $\mu(T)$ is conciliatory, we have $\mu(T)\circ\mathcal{U}_\xi\circ f=\mu(T)\circ\mathcal{U}_\xi\circ g$.
Therefore, $\mu s^*_\xi(T)\leq_1\mu s^*_\xi(V)$ via $g$.
\end{proof}

The above embedding of $(T_\eta(Q);\leq_h)$ into $(\mathbf{\Delta}^0_{1+\xi+\eta}(Q^\calN);\leq_{1+\xi})$ easily extends to an embedding of $(T^\sqcup_\eta(Q);\leq_h)$ into $(\mathbf{\Delta}^0_{1+\xi+\eta}(Q^\calN);\leq_{1+\xi})$.
If $F$ is of the form $T_0\sqcup T_1\sqcup\dots$ for some trees $(T_i)$, we use the symbol $s_\xi^\ast(F)$ to denote $s^*_\xi(T_0)\sqcup s^*_\xi(T_1)\sqcup\dots$.
We also adopt the similar convention for $s_\alpha$
Then, it is easy to see that Lemma \ref{mut} extends to any forests $T,V\in T^\sqcup_\eta(Q)$ as follows:
\[T\leq_hV\iff \mu s_\xi^\ast(T)\leq_1\mu s_\xi^\ast(V)\iff \mu s_\xi^\ast(T)\leq_{1+\xi}^\wadge\mu s_\xi^\ast(V)\iff \mu s_\xi^\ast(T)\leq_{1+\xi}\mu s_\xi^\ast(V).\]

The following small technical lemma will be used in the surjectivity proof.

\begin{Lemma}\label{lem:commutative}
Let $\beta\leq\alpha$ be ordinals.
For any $\mathbf{\Delta}^0_{1+\om^\beta}$-piecewise continuous function $g$, there is a $\mathbf{\Delta}^0_{1+\om^\alpha+\om^\beta}$-piecewise continuous function $h$ such that $g\circ\mathcal{U}_{\om^\alpha}\equiv_\con\mathcal{U}_{\om^\alpha}\circ h$.
\end{Lemma}

\begin{proof}
By the assumption, there is a $\mathbf{\Sigma}^0_{1+\om^\beta}$-partition $(X_i)_{i\in\om}$ of $\calN$ such that $g_i=g\upto X_i$ is continuous.
By Fact \ref{fact:conciliatory} (1), there is a total conciliatory extension $\hat{g}_i$ of $g_i$.
By Observation \ref{obs:measu} (1), $Y_i=\mathcal{U}_{\om^\alpha}^{-1}(X_i)$ is $\mathbf{\Sigma}^0_{1+\om^\alpha+\om^\beta}$, and $\hat{g}_i\circ\mathcal{U}_{\om^\alpha}$ is $\mathbf{\Sigma}^0_{1+\om^\alpha}$-measurable.
By Fact \ref{fact:conciliatory} (2), as $\mathcal{U}_{\om^\alpha}$ is $\mathbf{\Sigma}^0_{1+\om^\alpha}$-universal, for any $i$, there is a continuous function $h_i$ such that $\hat{g}_i\circ\mathcal{U}_{\om^\alpha}\equiv_\con\mathcal{U}_{\om^\alpha}\circ h_i$.
Then, define $h(x)=h_i(x)$ if $x\in Y_i$.
Clearly, $g\circ\mathcal{U}_{\om^\alpha}\equiv_\con\mathcal{U}_{\om^\alpha}\circ h$, and $h$ is $\mathbf{\Delta}^0_{1+\om^\alpha+\om^\beta}$-piecewise continuous.
\end{proof}


Note that $h$ in Lemma \ref{lem:commutative} is $\mathbf{\Delta}^0_{1+\om^\alpha+\om^\beta}$-piecewise total continuous, i.e., $h\upto Y_i=h_i\upto Y_i$ for a total continuous function $h_i$ on $\calN$.
Now, we prove our key lemma, which shows surjectivity of the map $T\mapsto\mu s_\xi^\ast(T)$ with respect to $\leq_{1+\xi}^\wadge$.

\begin{Lemma}\label{mut2}
Let $\xi,\eta$ be non-zero countable ordinals and  $Q$ be a bqo. Then for any $F\in\mathcal{T}^\sqcup_{\xi+\eta}(Q)$ there is $G\in\mathcal{T}^\sqcup_{\eta}(Q)$ such that $\mu(F)\equiv_{1+\xi}^\wadge\mu s^*_\xi(G)$.
\end{Lemma}

\begin{proof}
We first show the assertion for $\xi=\omega^\alpha$.
Let $F\in\mathcal{T}_{\xi+\eta}(Q)$ be given.
If $T=q$ for some $q\in Q$ we can take $G=\lambda x.q$.
If $F$ is not a tree, i.e., $F=F_0\sqcup F_1\sqcup\cdots$ for some $F_i$, then by induction hypothesis $\mu(F_i)\equiv^\wadge_{1+\xi}\mu s_\alpha(G_i)$ for some $G_i\in\mathcal{T}^\sqcup_{\eta}(Q)$.
Thus, we already have $\mu(F)\equiv_{1+\xi}\mu s_\alpha(G_0)\oplus\mu s_\alpha(G_1)\oplus\dots=\mu s_\alpha(G_0\sqcup G_1\sqcup\dots)$, so there is nothing to do.
Thus, we can assume that $F$ is a tree.


CASE 1. $F=A\cdot B$ for some tree $A$ and a forest $B$.

By induction hypothesis, there are $L,H\in\mathcal{T}_\eta^\sqcup(Q)$ such that $\mu(A)\equiv^\wadge_{1+\xi}\mu s_\alpha(L)$ and $\mu(B)\equiv^\wadge_{1+\xi}\mu s_\alpha(H)$.
We claim that 
\[\mu(F)=\mu(A)\cdot\mu(B)\equiv_{1+\xi}^\wadge\mu s_\alpha(L)\oplus \mu s_\alpha(H)=\mu s_\alpha(L\sqcup H).\]
The direction $\mu s_\alpha(L)\oplus \mu s_\alpha(H)\leq_{1+\xi}^\wadge\mu(A)\cdot\mu(B)$ is obvious.
We show the converse direction.
Let $\pi_0,\pi_1,J$ be as in Fact \ref{fact:concate}.
Then $(\mu(A)\cdot\mu(B))(x)$ is equal to $\mu(B)\circ\pi_1(x)$ if $x\in J$, and to $\mu(A)\circ\pi_0(x)$ if $x\not\in J$.
Let $g,h\in D_{1+\xi}^\wadge$ witness $\mu(A)\leq_{1+\xi}^\wadge\mu s_\alpha(L)$ and $\mu(B)\leq_{1+\xi}^\wadge\mu s_\alpha(H)$, respectively.
It is easy to see that $\mu(A)\cdot\mu(B)$ is reduced to $\mu s_\alpha(L)\oplus\mu s_\alpha(H)$ by the $D_{1+\xi}^\wadge$-function $x\mapsto 1h(\pi_1(x))$ on the open set $J$, and by $x\mapsto 0g(\pi_0(x))$ on the complement.
Since $J$ and its complements are $\mathbf{\Delta}^0_2$, and $2\leq 1+\xi$, this reduction is $\mathbf{\Delta}^0_{1+\xi}$-piecewise continuous.
This concludes the claim.
Hence we can take $G=L\sqcup H$.   

CASE 2. $F$ is a non-trivial singleton, i.e. $F=s_\gamma(T)$ for some tree $T$.

If $\gamma>\alpha$ then note that the language for $\mathcal{T}_\xi$ does not contain the symbol $s_\gamma$ (as $\xi=\omega^\alpha<\omega^\gamma$), and moreover, $\mathcal{T}_{\xi+\eta}(Q)=\mathcal{T}_{\omega^\alpha}\circ\mathcal{T}_\eta(Q)$.
This means that $F$ must be constructed from symbols $(\cdot,\sqcup,s_\beta)_{\beta<\alpha}$ and $s_\alpha(t)$ for terms $t$ in $T_\eta(Q)$.
Thus, if $s_\gamma$ occurs in $F$, then $F$ must be already contained in $\mathcal{T}_\eta(Q)$.
By definition, we have $s_\alpha s_\gamma(T)\equiv_hs_\gamma(T)$; hence we can take $G=F\in\mathcal{T}_\eta(Q)$.
Then, we have $F=s_\gamma(T)\equiv_hs_\alpha s_\gamma(T)=s_\alpha(F)=s_\alpha(G)$.

If $\gamma=\alpha$, then $F$ is already of the form $s_\alpha(T)$.
Again, note that the language for $\mathcal{T}_\xi=\mathcal{T}_{\omega^\alpha}$ does not contain the symbol $s_\alpha$, and moreover, $\mathcal{T}_{\xi+\eta}(Q)=\mathcal{T}_{\omega^\alpha}\circ\mathcal{T}_\eta(Q)$.
This again means that $F$ must be constructed from symbols $(\cdot,\sqcup,s_\beta)_{\beta<\alpha}$ and $s_\alpha(t)$ for terms $t$ in $T_\eta(Q)$.
Hence, we must have $T\in\mathcal{T}_\eta(Q)$.
Thus, we can take $G=T$.

Finally, let $\gamma<\alpha$.
By induction hypothesis, we have $\mu(T)\equiv_{1+\xi}^\wadge\mu s_\alpha(G)$ for some $G\in\mathcal{T}^\sqcup_\eta(Q)$.
This clearly implies that $\mu s_\alpha(G)\leq_{1+\xi}^\wadge\mu(F)$.
We show the converse direction.
Let $f\in D_{1+\om^\alpha}^\wadge$ witness $\mu(T)\leq_{1+\xi}^\wadge\mu s_\alpha(G)$.
We first assume that $G$ is a tree.
Then, we have
\[\mu(F)(x)=\mu(T)(\Uom{\gamma}(x))\leq_Q\mu s_\alpha(G)(f\circ \Uom{\gamma}(x))=\mu(G)(\Uom{\alpha}\circ f\circ \Uom{\gamma}(x)).\]
By Observation \ref{obs:measu} (2), $\mathcal{U}_{\omega^\gamma}\in D_{\omega^{\gamma+1}}\subseteq D_{\omega^\alpha}$, and we also have $f\in D_{\om^\alpha}$.
As $D_{\om^\alpha}$ is closed under composition, we have $f\circ\Uom{\gamma}\in D_{\om^\alpha}$, and therefore, $\Uom{\alpha}\circ f\circ \Uom{\gamma}$ is still $\mathbf{\Sigma}^0_{1+\om^\alpha}$-measurable by Observation \ref{obs:measu} (3).
By $\mathbf{\Sigma}^0_{1+\om^\alpha}$-universality of $\Uom{\alpha}$, there is a continuous function $h$ such that $\Uom{\alpha}\circ f\circ \Uom{\gamma}$ is $\con$-equivalent to $\Uom{\alpha}\circ h$.
Since $\mu(G)$ is conciliatory by Fact \ref{fact:conciliatory} (3), we obtain 
\[\mu(F)(x)\leq_Q\mu(G)(\Uom{\alpha}\circ h(x))=\mu s_\alpha(G)(h(x)).\]

This means that $\mu(F)\leq_1\mu s_\alpha(G)$.
If $G$ is a forest of the form $G_0\sqcup G_1\sqcup\dots$, then consider the set $X_n$ of all $x$ such that the first bit of $f\circ\Uom{\gamma}(x)$ is $n$.
Since $f\circ\Uom{\gamma}\in D_{\om^\alpha}$ as seen above, the set $X_n$ is $\mathbf{\Delta}^0_{1+\om^\alpha}$.
By replacing $f$ in the above argument with $f_n\colon x\mapsto nx$, it is straightforward to show that $\mu(F)$ is continuously reducible to $\mu s_\alpha(G_n)$ on $X_n$.
By combining these reductions, we get a $\mathbf{\Delta}^0_{1+\om^\alpha}$-piecewise continuous reduction from $\mu(F)$ to $\mu s_\alpha(G)$.
Consequently, $\mu(F)\equiv_{1+\xi}^\wadge\mu s_\alpha(G)$.

This concludes the proof for $\xi=\omega^\alpha$.
%
We now consider the general case $\xi=\omega^{\alpha_0}+\omega^{\alpha_1}+\cdots+\omega^{\alpha_m}$ for $m>0$.
Fix $F\in\mathcal{T}^\sqcup_\xi(Q)$, and let us consider $\eta_0=\omega^{\alpha_1}+\omega^{\alpha_2}+\dots+\omega^{\alpha_m}+\eta$.
Applying the above argument, there is $G_0\in\mathcal{T}^\sqcup_{\eta_0}(Q)$ such that $\mu(F)\equiv^\wadge_{1+\omega^{\alpha_0}}\mu s_{\alpha_0}(G_0)$.
Next, consider $\eta_1=\omega^{\alpha_2}+\dots+\omega^{\alpha_m}+\eta$.
Apply the above argument to $G_0$, there is $G_1\in\mathcal{T}^\sqcup_{\eta_1}(Q)$ such that $\mu(G_0)\equiv^\wadge_{1+\omega^{\alpha_1}}\mu s_{\alpha_1}(G_1)$.
Let us now consider
$\eta_i=\om^{\alpha_{i+1}}+\dots+\om^{\alpha_m}+\eta$.
By iterating the above procedure, we eventually obtain a sequence $G_0,G_1,\dots,G_m$ such that $G_i\in\mathcal{T}^\sqcup_{\eta_i}(Q)$ such that $\mu(G_{i-1})\equiv^\wadge_{1+\om^{\alpha_i}}\mu s_{\alpha_i}(G_i)$, where $G_{-1}=F$.
We now want to show $\mu(F)\equiv^\wadge_{1+\xi}\mu s^\ast_{\xi}(G_m)=\mu s_{\alpha_0}s_{\alpha_1}\dots s_{\alpha_m}(G_m)$.

To prove this, we claim that $\mu(F)\equiv^\wadge_{1+\om^{\alpha_0}+\om^{\alpha_1}}\mu s_{\alpha_0}s_{\alpha_1}(G_1)$.
For the forward direction, let $f\in D^\wadge_{1+\om^{\alpha_0}}$ witness $\mu(F)\leq_{1+\om^{\alpha_0}}^\wadge\mu s_{\alpha_0}(G_0)$ and $g\in D^\wadge_{1+\om^{\alpha_1}}$ witness $\mu(G_0)\leq_{1+\om^{\alpha_1}}^\wadge\mu s_{\alpha_1}(G_1)$.
First assume that both $G_0$ and $G_1$ are trees.
Then,
\begin{multline*}
\mu(F)(x)\leq_Q\mu s_{\alpha_0}(G_0)(f(x))=\mu(G_0)(\mathcal{U}_{\om^{\alpha_0}}\circ f(x))\\
\leq_Q\mu s_{\alpha_1}(G_1)(g\circ \mathcal{U}_{\om^{\alpha_0}}\circ f(x))=\mu (G_1)(\Uom{\alpha_1}\circ g\circ \mathcal{U}_{\om^{\alpha_0}}\circ f(x)).
\end{multline*}
As $f\in D_{1+\om^{\alpha_0}}$, the composition $\Uom{\alpha_0}\circ f$ is $\mathbf{\Sigma}^0_{1+\om^{\alpha_0}}$-measurable by Observation \ref{obs:measu} (3).
Similarly, as $g\in D_{1+\om^{\alpha_1}}$, $\Uom{\alpha_1}\circ g$ is $\mathbf{\Sigma}^0_{1+\om^{\alpha_1}}$-measurable.
Hence, $\Uom{\alpha_1}\circ g\circ\Uom{\alpha_0}\circ f$ is $\mathbf{\Sigma}^0_{1+\om^{\alpha_0}+\om^{\alpha_1}}$-measurable by Observation \ref{obs:measu} (1).
By $\mathbf{\Sigma}^0_{1+\om^{\alpha_0}+\om^{\alpha_1}}$-universality of $\Uom{\alpha_1}\circ\Uom{\alpha_0}$, there is a continuous function $h$ such that $\Uom{\alpha_1}\circ g\circ\Uom{\alpha_0}\circ f$ is $\con$-equivalent to $\Uom{\alpha_1}\circ\Uom{\alpha_0}\circ h$.
Since $\mu(G_1)$ is conciliatory by Fact \ref{fact:conciliatory} (3), we obtain
\[\mu(F)(x)\leq_Q\mu(G_1)(\Uom{\alpha_1}\circ\Uom{\alpha_0}\circ h(x))=\mu s_{\alpha_0}s_{\alpha_1}(G_1)(h(x)).\]
Hence, we have $F\leq_1\mu s_{\alpha_0}s_{\alpha_1}(G_1)$.
If $G_0$ and $G_1$ are forest, we need to decompose the domain according to the first bit of $f(x)$ and that of $g\circ\Uom{\alpha_0}\circ f(x)$.
These functions are $\mathbf{\Sigma}^0_{1+\om^{\alpha_0}+\om^{\alpha_1}}$-measurable, and so the decomposition is $\mathbf{\Delta}^0_{1+\om^{\alpha_0}+\om^{\alpha_1}}$.
Hence, $F$ is reducible to $\mu s_{\alpha_0}s_{\alpha_1}(G_1)$ by a $D_{1+\om^{\alpha_0}+\om^{\alpha_1}}^\wadge$-function.

For the converse direction, we similarly have a $D_{1+\om^{\alpha_0}}^\wadge$-function $f$ witnessing $\mu s_{\alpha_0}(G_0)\leq^\wadge_{1+\om^{\alpha_0}}\mu(F)$ and a $D_{1+\om^{\alpha_1}}^\wadge$-function $g$ witnessing $\mu s_{\alpha_1}(G_1)\leq^\wadge_{1+\om^{\alpha_1}}\mu(G_0)$.
We assume that both $G_0$ and $G_1$ are trees.
Then, we have
\[\mu s_{\alpha_0}s_{\alpha_1}(G_1)(x)=\mu s_{\alpha_1}(G_1)(\Uom{\alpha_0}(x))\leq_Q\mu(G_0)(g\circ\Uom{\alpha_0}(x)).\]
By Lemma \ref{lem:commutative}, there is a $D^\wadge_{1+\om^{\alpha_0}+\om^{\alpha_1}}$-function $h$ such that $g\circ\Uom{\alpha_0}$ is $\con$-equivalent to $\Uom{\alpha_0}\circ h$.
Then, as $\mu(G_0)$ is conciliatory, we now have
\[\mu(G_0)(g\circ\Uom{\alpha_0}(x))=\mu(G_0)(\Uom{{\alpha_0}}\circ h(x))=\mu{s_{\alpha_0}}(G_0)(h(x))\leq_Q\mu(F)(f\circ h(x)).\]
By combining the above two inequalities, we obtain $\mu s_{\alpha_0}s_{\alpha_1}(G_1)(x)\leq_Q\mu(F)(f\circ h(x))$.
Since $f\circ h$ is in $D_{1+\om^{\alpha_0}+\om^{\alpha_1}}^\wadge$ (as $D_{1+\om^{\alpha_0}+\om^{\alpha_1}}^\wadge$ is closed under composition), this witnesses $\mu s_{\alpha_0}s_{\alpha_1}(G_1)\leq^\wadge_{1+\om^{\alpha_0}+\om^{\alpha_1}}\mu(F)$.
If $G_0$ and $G_1$ are forest, as in the above argument, we have similar reductions on $\mathbf{\Delta}^0_{1+\om^{\alpha_0}+\om^{\alpha_1}}$ domains.
This concludes the proof of our claim.

We now apply this claim to the sequence $G_0,G_1,\dots,G_m$, where $G_m\in\mathcal{T}_\eta(Q)$.
Then, we eventually obtain $\mu(F)\equiv^\wadge_{1+\xi}\mu s^\ast_\xi(G_m)$.
Then, take $G=G_m$.
%
%
\end{proof}

\begin{proof}[Proof of Theorem \ref{mainth}.] The case $\xi=0$ coincides with the second assertion in Fact \ref{fact:km-main} while the case $\eta=0$ was considered in Lemma \ref{eta0}, so we assume that both $\xi,\eta$ are non-zero.
We show that $T\mapsto \mu s_\xi^\ast(T)$ induces an isomorphism $(\mathcal{T}^\sqcup_{\eta}(Q);\leq_h)\simeq(\mathbf{\Delta}^0_{1+\xi+\eta}(Q^\mathcal{N});\leq_{1+\xi})$.
Note that $T\in\mathcal{T}^\sqcup_{\eta}(Q)$ implies $s_\xi^\ast (T)\in\mathcal{T}^\sqcup_{\xi+\eta}(Q)$ by Observation \ref{t}, so $\mu s_\xi^\ast(T)\in\mathbf{\Delta}^0_{1+\xi+\eta}(Q^\mathcal{N})$ by Fact \ref{fact:mu-measu}.
By Lemma \ref{mut}, for all $T,V\in\mathcal{T}_{\eta}(Q)$ we have: $T\leq_hV$ iff   $\mu s^*_\xi(T)\leq_{1+\xi}\mu s^*_\xi(V)$ iff $\mu s^*_\xi(T)\leq_{1+\xi}^\wadge\mu s^*_\xi(V)$, so it suffices to show that for any $A\in\mathbf{\Delta}^0_{1+\xi+\eta}(Q^\calN)$ there is $G\in\mathcal{T}^\sqcup_{\eta}(Q)$ with $A\equiv_{1+\xi}^\wadge\mu s^*_\xi(G)$ (which clearly implies $A\equiv_{1+\xi}\mu s^*_\xi(G)$). By Fact \ref{fact:km-main} again, $A\equiv_1\mu (F)$ for some $F\in\mathcal{T}^\sqcup_{\xi+\eta}(Q)$.
By Lemma \ref{mut2}, there is $G\in\mathcal{T}^\sqcup_{\eta}(Q)$ such that $\mu(F)\equiv_{1+\xi}^\wadge\mu s^*_\xi(G)$. Thus, $A\equiv_{1+\xi}^\wadge\mu s^*_\xi(G)$ as desired.
\end{proof}

Indeed, the above proof shows that our main result holds for any qo $\preceq$ which is intermediate between $\leq_{1+\xi}^\wadge$ and $\leq_{1+\xi}$; that is, $(\mathcal{T}^\sqcup_{\eta}(Q);\leq_h)\simeq(\mathbf{\Delta}^0_{1+\xi+\eta}(Q^\mathcal{N});\preceq)$.
However, any nontrivial Borel amenable reducibility notion is induced by a class of the form ${D}_\xi^\mathcal{F}$ for some $\xi<\om_1$ and $\mathcal{F}$, and most natural classes $\mathcal{F}$ considered in Section \ref{wadge} satisfy this condition (recall that $D_\xi$ is the greatest one among level $\xi$ classes).
Consequently, our main result gives a complete combinatorial description of the structure of Borel $Q$-partitions w.r.t.~these Borel amenable reducibilities.

\section{Characterisation in terms of forests}\label{forest}

By Fact \ref{fact:km-main}, the map $\mu$ gives an isomorphism $(\mathcal{T}_{\xi+\eta}^\sqcup(Q);\leq_h)\simeq(\mathbf{\Delta}^0_{1+\xi+\eta}(Q^\calN);\leq_1)$.
On the one hand, we have a coarser qo $\leq_{1+\xi}$ on $\mathbf{\Delta}^0_{1+\xi+\eta}(Q^\calN)$, which induces a qo $\leq_h^{\xi}$ on $\mathcal{T}_{\xi+\eta}^\sqcup(Q)$ defined by $T\leq_h^{\xi}S$ iff $\mu(T)\leq_{1+\xi}\mu(S)$.
The quotient poset $(\mathcal{T}_{\xi+\eta}^\sqcup(Q);\leq_h^{\xi})$ is clearly isomorphic to $(\mathbf{\Delta}^0_{1+\xi+\eta}(Q^\calN);\leq_{1+\xi})$.
On the other hand, by Theorem \ref{mainth}, we have $(\mathcal{T}_\eta^\sqcup(Q);\leq_h)\simeq(\mathbf{\Delta}^0_{1+\xi+\eta}(Q^\calN);\leq_{1+\xi})$.
Therefore, we have a collection $(\leq_h^{\xi})_{\xi<\om_1}$ of induced qo's satisfying $(\mathcal{T}_\eta^\sqcup(Q);\leq_h)\simeq(\mathcal{T}_{\alpha+\eta}^\sqcup(Q);\leq_h^{\alpha})\simeq(\mathcal{T}_{\beta+\eta}^\sqcup(Q);\leq_h^{\beta})$.
%
%
Is it possible to characterise these induced qo's in terms of (natural operations on) labeled forests, without using the seemingly quite different Wadge-like reducibilities on $Q$-partitions?
In this section we address this question, and show the following theorem:

\begin{Theorem}\label{thm:main-syn}
For any ordinal $\xi<\om_1$, there is an endomorphism $r_\xi^\ast$ on $(\mathcal{T}_{\om_1}^\sqcup(Q);\leq_h)$ such that $r_\xi^\ast$ maps $\mathcal{T}^\sqcup_{\xi+\eta}(Q)$ into $\mathcal{T}_\eta^\sqcup(Q)$ for any $\eta$, and moreover, for any $T,S\in\mathcal{T}_{\om_1}^\sqcup$,
\[r_\xi^\ast(T)\leq_hr_\xi^\ast(S)\iff\mu(T)\leq_{1+\xi}\mu(S).\]
\end{Theorem}
To prove Theorem \ref{thm:main-syn}, we first recall that, in the proof of Lemma \ref{mut2}, given $F\in\mathcal{T}^\sqcup_{\xi+\eta}(Q)$, we explicitly defined $G\in\mathcal{T}^\sqcup_{\eta}(Q)$ such that $\mu(F)\equiv^\wadge_{1+\xi}\mu s_\xi^\ast(G)$.
This construction induces a map $F\mapsto G$.
To give a more explicit description of this map, we first consider the case $\xi=\om^\alpha$.
According to our proof of Lemma \ref{mut2}, the map $r_\alpha\colon F\mapsto G$ is defined in the following inductive manner:
\begin{enumerate}
\item $r_\alpha(q)=q$ for any $q\in Q$;
\item If $F=F_0\sqcup F_1\sqcup\cdots$ for some trees $F_i$ then $r_\alpha(F)=r_\alpha(F_0)\sqcup r_\alpha(F_1)\sqcup\cdots$;
\item If $F=T\cdot V$ for some singleton $T$ and forest $V$ then $r_\alpha(F)=r_\alpha(T)\sqcup r_\alpha(V)$;
\item if $F=s_\beta(T)$ for some tree $T$ then $r_\alpha(F)=r_\alpha(T)$ for $\alpha>\beta$, $r_\alpha(F)=T$ for $\alpha=\beta$, and $r_\alpha(F)=F$ for $\alpha<\beta$.
\end{enumerate}
For each $\xi$ we define a map $r^*_\xi:\mathcal{T}^\sqcup_{\omega_1}(Q)\to\mathcal{T}^\sqcup_{\omega_1}(Q)$ as follows: $r^*_0$ is the identity map $Id$, and if $\xi=\omega^{\alpha_0}+\cdots+\omega^{\alpha_m}>0$, $\alpha_0\geq\cdots\geq\alpha_m$, then $r^*_\xi=r_{\alpha_m}\circ\cdots\circ r_{\alpha_0}$.

\begin{Observation}\label{lem:sec-ret}
Let $\xi,\eta<\omega_1$ and  $Q$ be a bqo.
Then $r^*_\xi\circ s^*_\xi=Id$ and $r^*_\xi$ maps $\mathcal{T}^\sqcup_{\xi+\eta}(Q)$ into $\mathcal{T}^\sqcup_{\eta}(Q)$.
\end{Observation}

\begin{proof}
For the first assertion, we have $r_\alpha s_\alpha(T)=T$ by definition.
Thus, $r^\ast_\xi s^\ast_\xi=r_{\alpha_m}\dots r_{\alpha_0} s_{\alpha_0}\dots s_{\alpha_m}$ is the identity map.
The second assertion follows from the proof of Lemma \ref{mut2}.
\end{proof}
 
\begin{Lemma}\label{lem:endo}
For any $\xi<\om_1$, $r^\ast_\xi$ is an endomorphism on $(\mathcal{T}_{\om_1}^\sqcup(Q);\leq_h)$.
\end{Lemma}

\begin{proof}
It suffices to show that $T\leq_hS$ implies $r_\alpha(T)\leq_hr_\alpha(S)$.
We prove the assertion by induction on the complexity of $S,T\in\mathcal{T}_{\om_1}^\sqcup(Q)$.
The base case is trivial.
If $T=\sqcup_iT_i$ is a forest, then $T\leq_hS$ iff $T_i\leq_hS$ for all $i$, and by induction hypothesis, we have $r_\alpha(T_i)\leq_h r_\alpha(S)$ for any $i$.
Thus, $r_\alpha(T)=\sqcup_ir_\alpha(T_i)\leq_hr_\alpha(S)$.
If $S=\sqcup_iS_i$ is a forest, then $T\leq_hS$ iff $T\leq_hS_i$ for some $i$.
By induction hypothesis, this implies $r_\alpha(T)\leq_hr_\alpha(S_i)$ for some $i$; hence $r_\alpha(T)\leq_hr_\alpha(S)$.

Assume that $T=A\cdot B$ for some singleton $A$ and forest $B$, and $S$ is a singleton.
By definition of $\leq_h$, $A\cdot B\leq_h S$ implies $A,B\leq_hS$.
By induction hypothesis, we have $r_\alpha(A),r_\alpha(B)\leq_hr_\alpha(S)$, so $r_\alpha(T)=r_\alpha(A\cdot B)=r_\alpha(A)\sqcup r_\alpha(B)\leq_hr_\alpha(S)$.
Assume that $T$ is a singleton, and $S=C\cdot D$ for some singleton $C$ and forest $D$.
By definition of $\leq_h$, $T\leq_hC\cdot D$ implies either $T\leq_hC$ or $T\leq_hD$.
By induction hypothesis, $r_\alpha(T)\leq_hr_\alpha(C)$ or $r_\alpha(T)\leq_hr_\alpha(D)$; hence $r_\alpha(T)\leq_hr_\alpha(C)\sqcup r_\alpha(D)=r_\alpha(C\cdot D)=r_\alpha(S)$.
Assume that $T=A\cdot B$ and $S=C\cdot D$ for some singletons $A,C$ and forests $B,D$.
By definition, $A\cdot B\leq_hC\cdot D$ iff either $A\leq_hC$ and $B\leq_hC\cdot D$ or $A\cdot B\leq_hD$.
By induction hypothesis, either $r_\alpha(A)\leq_hr_\alpha(C)$ and $r_\alpha(B)\leq r_\alpha(C\cdot D)$ or $r_\alpha(A\cdot B)\leq_h r_\alpha(D)$.
In any case, $r_\alpha(A\cdot B)=r_\alpha(A)\sqcup r_\alpha(B)\leq_h r_\alpha(C)\sqcup r_\alpha(D)=r_\alpha(C\cdot D)$.

Finally, let us consider the case that both $S=s_\beta(U)$ and $T=s_\gamma(V)$ are singletons.
Assume that $S\leq_h T$, and we want to show that $r_\alpha s_\beta(U)\leq_hr_\alpha s_\gamma(V)$.
First consider the case $\beta=\gamma$.
Then, $S\leq_hT$ iff $U\leq_hV$.
By induction hypothesis, $r_\alpha(U)\leq_hr_\alpha(V)$.
If $\beta<\alpha$, then $r_\alpha s_\beta(U)=r_\alpha(U)\leq_h r_\alpha(V)=r_\alpha s_\gamma(V)$.
If $\beta=\alpha$, then $r_\alpha s_\beta (U)=U\leq_hV=r_\alpha s_\gamma(V)$.
If $\beta>\alpha$, then $r_\alpha s_\beta(U)=s_\beta(U)\leq_h s_\beta(V)=s_\gamma(V)=r_\alpha s_\gamma(V)$, since $U\leq_hV$ iff $s_\beta(U)\leq_h s_\beta(V)$ by definition.

For the case $\beta<\gamma$, the assumption $s_\beta(U)\leq_h s_\gamma(V)$ is equivalent to $U\leq_h s_\gamma(V)$.
By induction hypothesis, $r_\alpha(U)\leq_h r_\alpha s_\gamma(V)$.
If $\beta<\alpha$, then $r_\alpha s_\beta(U)=r_\alpha(U)\leq_h r_\alpha s_\gamma(V)$.
If $\beta=\alpha$, then $r_\alpha s_\beta(U)=U\leq_h s_\gamma(V)=r_\alpha s_\gamma(V)$ by $\gamma>\beta=\alpha$.
If $\beta>\alpha$, then $r_\alpha s_\beta(U)=s_\beta(U)\leq s_\beta s_\gamma(V)$ since $U\leq_hs_\gamma(V)$ iff $s_\beta(U)\leq_hs_\beta s_\gamma(V)$.
By $\beta<\gamma$, we have $s_\beta s_\gamma(V)\equiv_h s_\gamma(V)=r_\alpha s_\gamma(V)$, where the last equality follows from $\gamma>\alpha$.
Thus, $r_\alpha s_\beta(U)\leq_hr_\alpha s_\gamma(V)$.

For the case $\beta>\gamma$, the assumption $s_\beta(U)\leq_h s_\gamma(V)$ is equivalent to $s_\beta(U)\leq_h V$.
By induction hypothesis, $r_\alpha s_\beta(U)\leq_h r_\alpha (V)$.
If $\gamma<\alpha$, then $r_\alpha s_\beta(U)\leq_hr_\alpha(V)=r_\alpha s_\gamma(V)$.
If $\gamma=\alpha$, then $r_\alpha s_\beta(U)=s_\beta(U)\leq_h V=r_\alpha s_\gamma(V)$.
If $\gamma>\alpha$, then $r_\alpha s_\beta(U)=s_\beta(U)\leq_h V\leq_h s_\beta(V)=r_\alpha s_\gamma(V)$.
\end{proof}

\begin{proof}[Proof of Theorem \ref{thm:main-syn}.]
By Lemma \ref{lem:endo}, $r_\xi^\ast$ is an endomorphism.
Moreover, if $T\in \mathcal{T}^\sqcup_{\xi+\eta}(Q)$ then $r_\xi^\ast(T)\in\mathcal{T}^\sqcup_\eta(Q)$ by Observation \ref{lem:sec-ret}.
By Lemma \ref{mut}, for any $T,V\in\mathcal{T}_{\xi+\eta}^\sqcup(Q)$, $r_\xi^\ast(T)\leq_hr_\xi^\ast(V)$ if and only if $\mu s_\xi^\ast r_\xi^\ast(T)\leq_{1+\xi} \mu s_\xi^\ast r_\xi^\ast(V)$.
It is straightforward to check that, in our proof of Lemma \ref{mut2}, given $F\in\mathcal{T}^\sqcup_{\xi+\eta}(Q)$, $G$ is chosen as $r^\ast_\xi(F)$.
Thus, by Lemma \ref{mut2}, $\mu(T)\equiv_{1+\xi}^\wadge\mu s_\xi^\ast r^\ast_\xi(T)$, and the similar equivalence holds for $V$.
Consequently, $r_\xi^\ast(T)\leq_hr^\ast_\xi(V)$ if and only if $\mu(T)\leq_{1+\xi}\mu(V)$.
\end{proof}

\begin{Corollary}
For all ordinals $\xi,\eta<\om_1$ and bqo $Q$, $(\mathbf{\Delta}^0_{1+\eta}(Q^\calN);\leq_W)$ is a retract of $(\mathbf{\Delta}^0_{1+\xi+\eta}(Q^\calN);\leq_W)$.
\end{Corollary}

\begin{proof}
By Theorem \ref{thm:main-syn}, $r_\xi^\ast$ is an endomorphism on $(\mathcal{T}_{\om_1}^\sqcup;\leq_h)$.
Thus, by Fact \ref{fact:km-main}, the map $\mu(T)\mapsto \mu r_\xi^\ast(T)$ is well-defined on the Wadge degrees.
This induces a map $\tilde{r}_\xi^\ast\colon(\mathbf{\Delta}^0_{1+\xi+\eta}(Q^\calN);\leq_W)\to(\mathbf{\Delta}^0_{1+\eta}(Q^\calN);\leq_W)$ given by $[\mu(T)]\mapsto [\mu r_\xi^\ast(T)]$.
Then, $\tilde{s}_\xi^\ast\colon[\mu(T)]\mapsto[\mu s_\xi^\ast(T)]$ satisfies $\tilde{r}_\xi^\ast\tilde{s}_\xi^\ast=Id$ by Observation \ref{lem:sec-ret}; that is, $(\tilde{s}^\ast_\xi,\tilde{r}^\ast_\xi)$ is a section-retraction pair.
\end{proof}

\section{Open question}

One of the most important open questions in bqo-Wadge theory seems to find reasonable generalisations of Theorems \ref{char1} and \ref{mainth} for a certain class of non-Borel functions (under some set-theoretic assumption).
For $Q=2$, most known results have been straightforwardly extended to non-Borel sets under the axiom of determinacy.
The reason why such an extension is possible is because the two-point space $2$ is too easy, so its Wadge degree structure is completely determined by Wadge's lemma, Martin-Monk's lemma, and Steel-van Wesep's theorem (cf.~\cite{cabal12}).
The bqo analogue of Wadge's lemma and Martin-Monk's lemma is van Engelen-Miller-Steel's theorem, and a bqo analogue of Steel-van Wesep's theorem is also known.
These theorems are readily extended to non-Borel $Q$-partitions under a certain set-theoretic assumption, cf.~\cite{Blo14,km18}.
However, contrary to the case $Q=2$, these theorems are far from characterising the Wadge degree structure even for $Q=3$.
Borel bqo-Wadge theory has played an significant role for unveiling hidden structures on the Wadge degrees which cannot be recognised by the Wadge theory for $Q=2$.
We expect that extending bqo-Wadge theory to non-Borel functions would lead us to new ideas revealing more deep structures in Wadge theory, and also to new constructions in wqo/bqo theory.

\section*{Acknowledgements}
The first-named author was partially supported by JSPS KAKENHI Grant 19K03602, 15H03634, and the JSPS Core-to-Core Program (A. Advanced Research Networks).

\end{document}